\documentclass[10pt,letterpaper,twoside]{article}
\usepackage{amsfonts} 
\usepackage{amssymb}  
\usepackage{graphicx} 
\usepackage[all]{xy}
\usepackage{amsthm}
\usepackage{amsmath}
\usepackage{newlfont}

\newtheorem{thm}{Theorem}[section]
\newtheorem{cor}[thm]{Corollary}
\newtheorem{lem}[thm]{Lemma}
\newtheorem{prop}[thm]{Proposition}
\theoremstyle{definition}
\newtheorem{Nota}[thm]{Remark}

\newcommand{\ve}{{\textbf{V}}}
\newcommand{\en}{\longrightarrow}
\newcommand{\id}{\mathrm{id}}
\newcommand{\Id}{\mathrm{Id}}
\newcommand{\I}{\mathrm{I}}

\newcommand{\be}{\mathrm{b}}
\newcommand{\de}{\mathrm{d}}
\newcommand{\ce}{\mathrm{c}}

\newcommand{\hbx}{\hfill\square}

\newcommand{\Obj}{\mathrm{Obj}}

\newcommand{\rc}{\mathcal{C}}
\newcommand{\rcs}{\mathcal{C}^{\textsf{S}_0}}
\newcommand{\s}{\textsc{S}}
\newcommand{\h}{\overline{\mathrm{h}}}
\newcommand{\Ce}{\mathrm{C}}

\begin{document}

\title { Braided bialgebras in a generated monoidal Ab-category}
\author{Ra\'ul A. P\'erez and Carlos Prieto}









\maketitle

\begin{abstract} We start from any small strict monoidal braided Ab-category
and extend it to a monoidal nonstrict braided Ab-category which
contains braided bialgebras. The objects of the original category
turn out to be modules for these bialgebras.
\end{abstract}



\setcounter{section}{-1}
\section{Introduction}
The notion of bialgebras and Hopf algebras in braided categories
were introduced by S. Majid in \cite{Mj1}. He has considered a
braided monoidal (tensor) category and in the usual definitions of
algebras, coalgebras, bialgebras and Hopf algebras, he has replaced
the flip by the braiding in the obvious way. Majid has called
braided bialgebra to a bialgebra in a braided category. Refer to
\cite{Ka95} and \cite{Tu94} for generalities in braided monoidal
categories and to \cite{Mj1}, \cite{Ta1} and \cite{Ta2} for the
definition and results in the theory of braided bialgebras and
braided Hopf
algebras.\\
 The purpose of this paper is to present a construction in which
 starting from a small braided monoidal Ab-category $\rc$ and an infinite
 set $\s_0$ we create a new monoidal braided category $\rcs$ that contains
 the original $\rc$ as a subcategory, and more important, it contains objects with bialgebra
 structure, in such a way that the objects of the original category
$\rc$ are modules for these bialgebras.
Remember that a category $\rc$ is said to be an Ab-category (also called
\emph{pre-abelian} category; cf.\cite{MacL}) if for any pair of objects
$V$, $W$ its set of morphisms $\hom(V,W)$ is an additive abelian
group and the composition of morphisms is bilinear. In the context
of monoidal Ab-categories we shall assume that the tensor product of
morphisms is bilinear. For the construction we proceed as follows.
Section \ref{seccion1} is divided into two parts; in the first part,
out of any small Ab-category $\rc$ and any set $\s_0$, we construct
the new category $\rcs$ which is also an Ab-category. In the second
part we assume that $\rc$ is strict monoidal and that $\s_0$ is
infinite, and so, we extend the monoidal structure to $\rcs$.
However, the extended monoidal structure is not strict, so we have
to work with associative constraints, and left and right units. In
the second part we show how to extend a braiding and a twist from
$\rc$ to $\rcs$. Since the new category is nonstrict monoidal, we
need to define algebras, coalgebras, bialgebras and modules in this
case. This is easily made if in the categorical definitions of the
latter notions we substitute the equalities by an equivalent
relation in the set of morphisms of $\rcs$. Roughly
speaking, we declare two morphisms of $\rcs$ related if their
domains and codomains are related by associativity and/or units.
This relation, obviously, agree with the identity if the category is
stric; this is  explained in detail at the end of section
\ref{seccion1}. We start section \ref{section2} defining
algebras, coalgebras, bialgebras and modules in nonstric monoidal
categories in general, and then state and prove the main theorem
(Theorem \ref{principal}) of this paper. Throughout the proof we use
graphical calculus as explained in \cite{Ka95} and \cite{Tu94}. This
work is influenced by  Yetter's paper \cite{Ye90}.


\section{The category $\rcs$}\label{seccion1}
Let $\mathcal{C}$ be a small Ab-category. We shall denote by
$\Obj(\rc)$ and $\mathcal{H}$ the sets of its objects and morphisms,
respectively. We are going to associate to $\mathcal{C}$ a new
category $\mathcal{C}^{\textsf{S}_0}$ as follows. Let us take a
fixed set $\textsc{S}_0$ and consider the set
$\mathcal{M}(\s_0,\Obj(\rc))=\{\xymatrix{\s_0\supset
\s_f\ar[r]^-{f}& \Obj(\rc)}\}$, where $\s_f$ is any subset of $\s_0$
and $f$ is a set-theoretical function. The \emph{objects} of
$\mathcal{C}^{\textsf{S}_0}$ will be the elements of
$\mathcal{M}(\s_0,\Obj(\rc))$. Let $f:\s_f\en \Obj(\rc)$ and
$g:\s_g\en \Obj(\rc)$ be two objects. A \emph{morphism} $F:f\en g$
will be a two-variable function $F:\s_f\times \s_g\en \mathcal{H}$
such that:
\begin{itemize}\label{def. de morfismo en cs}
\item[(i)] $F(x,y):f(x)\en g(y)$, for all $(x,y)\in \s_f\times
\s_g$.
\item[(ii)] If $\s_g$ is infinite, then for each $x\in \s_f$
there exists a finite set $S^{F}_x\subset \s_g$, such that
$F(x,y)=0$ if $y\in \s_g-S^{F}_x$.
\end{itemize}

Let $f:\s_f\en \Obj(\rc)$, $g:\s_g\en \Obj(\rc)$, and $h:\s_h\en
\Obj(\rc)$ be objects, and $F:f\en g$, $G:g\en h$ be morphisms. Define
$G\circ F:f\en h$ as the function $G\circ F:\s_f\times \s_h\en
\mathcal{H}$ given by:
\begin{equation}\label{composition}
(G\circ F)(x,y)=\sum_{z\in \s_g}G(z,y)\circ F(x,z)
\end{equation}
for $x\in \s_f$ and $y\in \s_h$. This sum is always finite.
Indeed, if we write $S^{F}_x=\{z_1,...,z_k\}$, then the sum
becomes
\begin{equation}\label{composition-sum}
(G\circ F)(x,y)=\sum_{i=1}^{k}G(z_i,y)\circ F(x,z_i)
\end{equation}
It is clear that the function $G\circ F$ satisfies condition
(i). Besides, if $y\notin S^{G}_{z_1}\cup ...\cup S^{G}_{z_k}$,
then $G(z_i,y)=0$ for $1\leq i\leq k$, so if we choose
$S_x^{(G\circ F)}=S^{G}_{z_1}\cup ...\cup S^{G}_{z_k}$, then we have
$y\in \s_h-S_x^{(G\circ F)}$, thus $(G\circ F)(x,y)=0$. Therefore
$G\circ F$ also satisfies condition (ii).

For any $f:\s_f\en \Obj(\rc)$ define $\Id_{f}:f\en f$ as the
function $\Id_{f}:\s_f\times \s_f\en \mathcal{H}$, given by
$\Id_{f}(x,y)=\delta_{x,y}\id_{f(x)}:f(x)\en f(y)$ for $(x,y)\in
\s_f\times \s_f$. For $G:f\en g$ one has
\begin{equation}\label{funct-1}
\begin{split}
(G\circ \Id_f)(x,y)&=\sum_{z\in \s_f}G(z,y)\circ \Id(x,z) \\
&=\sum_{z\in \s_f}G(z,y)\circ \delta_{x,z}\id_{f(x)}\\
&=G(x,y)
\end{split}
\end{equation}
Therefore $G\circ \Id_{f}=G$. Analogously $\Id_{g}\circ G=G$
 for any morphism $G:f\en g$.

Furthermore this operation is associative. Indeed, if $F:f\en g$,
$G:g\en h$, and $H:h\en i$, then
\begin{equation}\label{funct-2}
\begin{split}
((H\circ G)\circ F)(w,z)&=\sum_{x\in \s_g}(H\circ G)(x,z)\circ
F(w,x)\\
&=\sum_{x\in \s_g}\sum_{y\in \s_h}(H(y,z)\circ G(x,y))\circ
F(w,x)\\
&=\sum_{y\in \s_h}H(y,z)\circ (\sum_{x\in \s_g}G(x,y)\circ
F(w,x))\\
&=\sum_{y\in \s_h}H(y,z)\circ(G\circ F)(w,y)\\
&=(H\circ (G\circ F))(w,z)
\end{split}
\end{equation}

Hence we have proved that $\mathcal{C}^{\textsf{S}_0}$ is a
category.
 If for two morphisms $F,G:f\en g$ we define the function
$(F+G)(x,y)=F(x,y)+G(x,y)$, which trivially satisfies conditions (i)
and (ii), we see that $\rcs$ is also an  Ab-category. The following
proposition proves that the direct sum of certain
collections of objects in $\rcs$ is defined.

\begin{prop}
Let $\{f_i:\s_i\en \Obj(\rc)\}_{i\in \mathcal{I}}$ be any
collection of functions such that the sets $\s_i$, $i\in \mathcal{I}$,
are pairwise disjoint subsets of $\s_0$. Then $(f:\coprod_{i\in
\mathcal{I}} \s_i\en \Obj(\rc), J_i)$, where $f|_{\s_i}=f_i$ and
$J_k:\s_k \times \coprod_{i\in \mathcal{I}} \s_i\en \mathcal{H}$
is given by $J_k(x,y)=\delta_{xy}\id_x:\s_k\en \coprod_{i\in
\mathcal{I}} \s_i$, is the coproduct of $\{f_i:\s_i\en
\Obj(\rc)\}_{i\in \mathcal{I}}$ in $\mathcal{C}^{\textsf{S}_0}$.
\end{prop}

\begin{proof} Suppose we are given an object $g:\s_g\en \Obj(\rc)$
and a family of morphisms $T_i:(f_i:\s_i\en \Obj(\rc))\en
(g:\s_g\en \Obj(\rc))$. Define $T:(f:\coprod_{i\in \mathcal{I}}
\s_i\en \Obj(\rc))\en (g:\s_g\en \Obj(\rc))$ to be the function
$T(t,y)=T_k(t,y):f(t)\en g(y)$ if $t\in \s_k$. Then if $x\in \s_k$
and $y\in \s_g$, we have
\begin{equation}
\begin{split}
(T\circ J_k)(x,y)&=\sum_{t\in \coprod\s_i} T(t,x)\circ
J_k(x,t)\\&=\sum_{t\in \coprod\s_i} \delta_{xt}T(t,y)\circ
\id_x\\&=T(x,y)\\&=T_k(x,y)
\end{split}
\end{equation}
 The last equality also shows the uniqueness of $T$.
\end{proof}
In particular we have the following.

 \begin{cor}
 If $\emptyset \neq \s_f\subset \s_0$, then any object $f:\s_f\en
 \Obj(\rc)$ is isomorphic to the direct sum of the objects
  \{$f|_{\{x\}}:\{x\}\en \Obj(\rc)\}_{x\in \s}$.
 \end{cor}

 Let us suppose now that the category $\rc$ is strict monoidal and that the set $\s_0$ is infinite. In what follows we shall endow
 $\mathcal{C}^{\textsf{S}_0}$ with a monoidal structure extending the
 one given
 in $\rc$. However, as we shall see, the structure that we define is not strict in general.

We start by defining the tensor product of objects and morphisms and a unit object. Next
  we define the associative constraint $A$, the left and right units $L$ and
 $R$, and finally we prove that they satisfy the required conditions.

First, we fix once and for all a bijection $\gamma:\s_0\times \s_0\en \s_0$.
 Given two objects $f:\s_f\en \Obj(\rc)$ and $g:\s_g\en
\Obj(\rc)$, define $f\otimes g$ by the following composite
\begin{equation} \xymatrix@C=0.8cm{
f\otimes g:\gamma(\s_f\times
\s_g)\ar[r]^-{\gamma^{-1}|}&\s_f\times \s_g\ar[r]^-{f\times
g}&\Obj(\rc)\times \Obj(\rc)\ar[r]^-{\otimes}&\Obj(\rc)}.
\end{equation}
Choose any point $\ast$ en $\s_0$ and define $\I:\{\ast \}\en
\Obj(\rc)$ by $\I(\ast)=\I\in \Obj(\rc)$.

Now, for two morphisms
$F:f\en f'$, $G:g\en g'$, and a point $(z,z')\in \gamma(\s_f\times
\s_g)\times \gamma(\s_{f'}\times \s_{g'})$, define $F\otimes
G:f\otimes g\en f'\otimes g'$ by
\begin{equation}
(F\otimes G)(z,z'):=F(x_z,x'_{z'})\otimes
G(y_z,y'_{z'}):f(x_z)\otimes g(y_z)\en f'(x'_{z'})\otimes
g'(y'_{z'}),
\end{equation}
where $\gamma^{-1}(z)=(x_z,y_z)\in \s_f\times \s_g$ and
$\gamma^{-1}(z')=(x'_{z'},y'_{z'})\in \s_{f'}\times \s_{g'}$ are
the pairs such that $(f\otimes g)(z)=f(x_z)\otimes g(y_z)$ and
$(f'\otimes
g')(z')=f'(x'_{z'})\otimes g'(y'_{z'})$.

It is clear that $\gamma(S_{x_z}^F\times S_{y_z}^G)\subset
\gamma(\s_{f'}\times \s_{g'})$ is a finite set and that if $z'\in
\gamma(\s_{f'}\times \s_{g'})-\gamma(S_{x_z}^F\times S_{y_z}^G)$,
then $\gamma^{-1}(z')\notin S^{F}_{x_z}\times S^{G}_{y_z}$. Hence,
either $x'_{z'}\notin S_{x_z}^F$ or $y'_{z'}\notin S_{y_z}^G$ and
so $(F\otimes G)(z,z')=0$ if $z'\notin
\gamma(S^{F}_{x_z}\times S^{G}_{y_z})$.

 Before we define the associative constraint $A$, we shall adopt the
 following notation.
If, for example, $v\in
\gamma(\gamma(\s_f\times \s_g)\times \s_h)$, then we write
$$(\gamma^{-1}\times \id)\gamma^{-1}(v)=((x_v,y_v),z_v)\in \s_f\times
\s_g\times \s_h.$$
Here, $\gamma(x_v,y_v)$ is the unique element in $\gamma(\s_f\times
\s_g)\subset \s_0$ such that $\gamma(\gamma(x_v,y_v),z_v)=v$. In
other words, the inner parentheses will indicate the place from left to right of the
second $\gamma^{-1}$ in the composition $(\gamma^{-1}\times
\id)\gamma^{-1}$. Analogously for $w\in
\gamma(\s_f\times \gamma(\s_g\times \s_h))$ we write
$$(\id\times \gamma^{-1})\gamma^{-1}(w)=(x_w,(y_w,z_w))\in \s_f\times \s_g\times
\s_h.$$
When there is no risk of confusion we drop the inner
parentheses and  simply write $(\gamma^{-1}\times
\id)\gamma^{-1}(v)=(x_v,y_v,z_v)$ and $(\id\times
\gamma^{-1})\gamma^{-1}(w)=(x_w,y_w,z_w)$. In the same way, if
for example $v\in \s_{(f\otimes (g\otimes h))\otimes i}=
\gamma(\gamma(\s_f\times \gamma(\s_g\times \s_h))\times \s_i)$, then we
write
$$(\id\times \gamma^{-1}\times \id)(\gamma^{-1}\times
\id)\gamma^{-1}(v)=((x_v,(y_v,z_v)),t_v)\in \s_f\times \s_g\times
\s_h\times \s_i,$$
or $(\id\times \gamma^{-1}\times
\id)(\gamma^{-1}\times \id)\gamma^{-1}(v)=(x_v,y_v,z_v,t_v)$, etc.

With this notation, we have
\begin{equation}
\begin{split}((F\otimes G)\otimes H)(v,w)&=(F\otimes
G)((x,y)_v,(x,y)_w)\otimes H(z_v,z_w)\\
&=F(x_v,x_w)\otimes G(y_v,y_w)\otimes H(z_v,z_w).\end{split}
\end{equation} Let us define $A_{f,g,h}:(f\otimes g)\otimes h\en f\otimes (g\otimes h)$
by
\begin{equation}
\xymatrix@C=0.7cm{A_{f,g,h}(v,w)=\delta^{v,w}_{x;y;z}
\id_{f(x_{v})\otimes g(y_{v})\otimes h(z_v)}:((f\otimes g)\otimes
h)(v)\ar[r]&(f\otimes (g\otimes h))(w)},
\end{equation}
where again, in order to shorten the notation,
$\delta^{v,w}_{x;y;z}$ stands for $\delta_{x_{v},x_w}
\delta_{{y_{v},y_{w}}} \delta_{{z_v,z_{w}}}$. It is easy to see that
the inverse of $A_{f,g,h}$ is given by
\begin{equation}
\xymatrix@C=0.7cm{A^{-1}_{f,g,h}(w,v)=\delta^{w,v}_{x;y;z}
\id_{f(x_{w})\otimes g(y_{w})\otimes h(z_w)}:(f\otimes (g\otimes
h))(w)\ar[r]&((f\otimes g)\otimes h)(v)}.
\end{equation}

Now we define the right unit $R_f:f\otimes \I\en f$. For any object
$f$, the object $f\otimes \I$ is expressed by the composite
\begin{equation} \xymatrix@C=0.8cm{f\otimes
\I:\gamma(\s_f\times \{\ast\})\ar[r]^-{\gamma^{-1}|}&\s_f\times
\{\ast\}\ar[r]^-{f\times \I}&\Obj(\rc)\times
\Obj(\rc)\ar[r]^-{\otimes}&\Obj(\rc)}\,.
\end{equation}

For $z\in \gamma(\s_f\times \{\ast\})$, we write
$\gamma^{-1}(z)=(x_z,\ast)\in \s_f\times \ast$ and define
$R_f:f\otimes \I\en f$ by
\begin{equation}
R_f(z,x)=\delta_{x_z,x}\id_{f(x_z)}:(f\otimes \I)(z)=f(x_z)\en
f(x)\end{equation}
for $(z,x)\in \gamma(\s_f\times \{\ast\})\times
\s_f$. It is easy to see that $R_f$ is an isomorphism with inverse
$R^{-1}_f:f\en f\otimes \I$ given by the function\
\begin{equation}R^{-1}_f(x,z)=\delta_{x,x_z}\id_{f(x)}:f(x)\en
(f\otimes \I)(z)=f(x_z)\,.
\end{equation}
In the same way we define
the left unit $L_f:\I\otimes f\en f$, that is, if $z\in
\gamma(\ast\times \s_f)$, then we write $\gamma^{-1}(z)=(\ast,x_z)$ and
define
\begin{equation}
L_f(z,x)=\delta_{x_z,x}\id_{f(x_z)}:(\I\otimes f)(z)=f(x_z)\en
f(x)\end{equation}
The inverse of $L_f$ is given by
\begin{equation}L^{-1}_f(x,z)=\delta_{x,x_z}\id_{f(x)}:f(x)\en (\id\otimes
f)(z)=f(x_z) \,.
\end{equation}

\begin{thm}\label{teorema: construccion rcs} The category
  $\mathcal{C}^{\textsf{S}_0}$ is a monoidal category
with tensor product of objects and morphisms, associative
constraint, and right and left units as we have just defined.
\end{thm}

We divide the {\em proof} into four lemmas.

\begin{lem} If $F:f\en f'$, $F':f'\en f''$, $G:g\en g'$, and $G':g'\en
g''$ are morphisms in $\mathcal{C}^{\textsf{C}_0}$, then
\begin{itemize}
\item[\em{(i)}] $(F'\otimes G')\circ (F\otimes G)=(F'\circ F)\otimes
(G'\circ G)$ and
\item[{\em (ii)}] $\Id_f\otimes \Id_g=\Id_{f\otimes g}$.
\end{itemize}
\end{lem}

\begin{proof}
 (i) For $z\in \gamma(\s_f\times \s_g)$ and $z''\in \gamma(\s_{f''}\times \s_{g''})$
  we have
\begin{equation}
\begin{split}
((F'\otimes G')\circ (F\otimes G))(z,z'')&=\!\!\!\!\!\sum_{z'\in
\gamma(\s_{f'}\times \s_{g'})}\!\!\!\!\!(F'\otimes G')(z',z'')\circ
(F\otimes G)(z,z')\\
&=\!\!\!\!\!\sum_{z'\in \gamma(\s_{f'}\times
\s_{g'})}\!\!\!\!\!(F'(x'_{z'},x''_{z''})\circ F(x_z,x'_{z'})) \\
&\quad \otimes
(G'(y'_{z'},y''_{z''})\circ G(y_z,y'_{z'}))\\
&=(\sum_{x'\in \s_{f'}}F'(x',x''_{z''})\circ F(x_z,x'))\\
& \quad \otimes
(\sum_{y'\in \s_{g'}}G'(y',y''_{z''})\circ G(y_z,y'))\\
&=(F'\circ F)(x_z,x''_{z''})\otimes (G'\circ G)(y_z,y''_{z''})\\
&=((F'\circ F)\otimes (G'\circ G))(z,z'')
\end{split}
\end{equation}
The third equality follows from the fact that $\gamma$ establishes a
bijection between $\s_{f'}\times \s_{g'}$ and
$\gamma(\s_{f'}\times
\s_{g'})$.

(ii) For $z,z'\in \gamma(\s_f\times \s_g)$ we have
\begin{equation}
\begin{split}(\Id_f\otimes \Id_g)(z,z')&=\Id_f(x_z,x_{z'})\otimes
\Id_g(y_z,y_{z'})\\
&= \delta_{x_z,x_{z'}}\id_{f(x_z)} \otimes
\delta_{y_z,y_{z'}}\id_{g(y_z)}\\
&= \delta_{z,z'}\id_{f(x_z)} \otimes \id_{g(y_z)}\\
&= \delta_{z,z'}\id_{f(x_z)\otimes g(y_z)}\\
&=\delta_{z,z'}\id_{(f\otimes g)(z)}\\
&=\Id_{f\otimes g}(z,z')
\end{split}
\end{equation}
\end{proof}

\begin{lem}
The associative constraint $A$ defined above is a natural
isomorphism that satisfies the Pentagonal Axiom.
\end{lem}

\begin{proof} We already saw that $A$ is an isomorphism. To show that
  it is natural, we have, on the one hand,
\begin{equation}
\begin{split}
((F\otimes (G\otimes H))\circ A_{f,g,h})(v,w')&=\!\!\!\!\!\sum_{w\in
\gamma(\s_f\times \gamma(\s_g\times \s_h))}\!\!\!\!\!(F\otimes (G\otimes
H))(w,w')\circ\\
& \qquad\qquad\circ A_{f,g,h}(v,w)\\
&=\!\!\!\!\!\sum_{w\in \gamma(\s_f\times \gamma(\s_g\times
\s_h))}\!\!\!\!\!(F(x_w,x_{w'})\otimes G(y_w,y_{w'})\otimes \\
&\qquad\qquad \otimes H(z_w,z_{w'}))\circ \delta^{v,w}_{x;y;z}
\id_{f(x_{v})\otimes
g(y_{v})\otimes h(z_v)}\\
&=F(x_v,x_{w'})\otimes G(y_v,y_{w'}) \otimes H(z_v,z_{w'})\,.
\end{split}
\end{equation}
On the other hand, we have
\begin{equation}
\begin{split}
(A_{f',g',h'}\circ ((F\otimes G)\otimes H))(v,w')&=\sum_{v'\in
\gamma(\gamma(\s_{f'}\times \s_{g'})\times
\s_{h'})}\!\!\!\!\!\!\!\!\!\!\!\!\!\!\!\!A_{f',g',h'}(v',w')\circ ((F\otimes G)\otimes\\
&\quad H)(v,v')\\
&=\sum_{v'\in \gamma(\gamma(\s_{f'}\times \s_{g'})\times
\s_{h'})}\!\!\!\!\!\!\!\!\!\!\!\!\!\!\!\!\delta^{v',w'}_{x';y';z'}
\id_{f'(x'_{v'})\otimes
g'(y'_{v'})\otimes h'(z'_{v'})}\circ \\
&\quad \circ (F(x_v,x'_{v'})\otimes G(y_v,y'_{v'})\otimes H(z_v,z'_{v'}))\\
&=F(x_v,x_{w'})\otimes G(y_v,y_{w'}) \otimes H(z_v,z_{w'}).
\end{split}
\end{equation}
Therefore, $A_{f',g',h'}\circ ((F\otimes G)\otimes H)=(F\otimes
(G\otimes H))\circ A_{f,g,h}$, so $A$ is natural.

Let us prove now that $A$ satisfies the Pentagonal Axiom. Set
$M(s,w)=((\id_f\otimes A_{g,h,i})\circ A_{f,g\otimes h,i}\circ
(A_{f,g,h}\otimes \id_i))(s,w)$. For $s\in
\gamma(\gamma(\gamma(\s_f\times \s_g)\times \s_h)\times \s_i)$ and
$w\in \gamma(\s_f\times \gamma(\s_g\times \gamma(\s_h\times
\s_i)))$, we have
\begin{equation}
\begin{split}
M(s,w)&=\!\!\!\!\!\sum_{\substack{u\in \s_{f\otimes ((g\otimes
h)\otimes i)}\\
v\in \s_{f\otimes ((g\otimes h)\otimes i)}}}\!\!\!\!\!(\id_f\otimes
A_{g,h,i})(v,w)\circ A_{f,g\otimes h,i}(u,v)\circ
(A_{f,g,h}\otimes \id_i)(s,u)\\
&=\!\!\!\!\!\sum_{\substack{u\in \s_{f\otimes ((g\otimes
h)\otimes i)}\\
v\in \s_{f\otimes ((g\otimes h)\otimes
i)}}}\!\!\!\!\!(\delta_{x_v,x_w}\id_{f(x_v)}\otimes
\delta^{v,w}_{y;z;t}\id_{g(y_v)\otimes h(z_v)\otimes i(t_v)})\circ
\delta^{u,v}_{x;y;z;t}\\
&\quad \id_{f(x_u)\otimes g(y_u)\otimes h(z_u)\otimes i(t_u)}\circ
(\delta^{s,u}_{x;y;z}\id_{f(x_s)\otimes g(y_s)\otimes
h(z_s)}\otimes \delta_{t_s,t_u}\id_{i(t_s)}) \\
&=\delta^{s,w}_{x;y;z;t}\id_{f(x_s)\otimes g(y_s)\otimes
h(z_s)\otimes i(t_s)}.
\end{split}
\end{equation}
Set $N(s,w)=(A_{f,g,h\otimes i}\circ A_{f\otimes g,h,i})(s,w)$.
Then,
\begin{equation}
\begin{split}
N(s,w)&=\!\!\!\!\!\sum_{r\in \s_{(f\otimes g)\otimes (h\otimes
i)}}\!\!\!\!\!(A_{f,g,h\otimes i})(r,w)\circ (A_{f\otimes g,h,i})(s,r)\\
&=\!\!\!\!\!\sum_{r\in \s_{(f\otimes g)\otimes (h\otimes
i)}}\!\!\!\!\!\delta^{r,w}_{x;y;z;t}\id_{f(x_r)\otimes g(y_r)\otimes
h(z_r)\otimes i(t_r)}\circ
\delta^{s,r}_{x;y;z;t}\id_{f(x_s)\otimes g(y_s)\otimes
h(z_s)\otimes i(t_s)}\\
&=\delta^{s,w}_{x;y;z;t}\id_{f(x_s)\otimes g(y_s)\otimes
h(z_s)\otimes i(t_s)}.
\end{split}
\end{equation}
Thus $M(s,w)=N(s,w)$ and so, $A$ satisfies the Pentagonal Axiom.
\end{proof}

\begin{lem} The right unit $R$ and the left unit $L$ are natural isomorphisms.
\end{lem}

\begin{proof} We already saw that $R_f$ is an isomorphism. For $z\in
\gamma(\s_f\times \ast)$ and $x'\in \s_{f'}$ we have
\begin{equation}
\begin{split}
(F\circ R_f)(z,x')&=\sum_{x\in \s_f} F(x,x')\circ R_{f}(z,x)\\
&=F(x,x')\circ \delta_{x_z,x'}\id_{f(x_z)}\\
&=F(x_z,x').
\end{split}
\end{equation}
On the other hand
\begin{equation}
\begin{split}
(R_{f'}\circ (F\otimes \Id_{\I}))(z,x')&=\sum_{z'\in
\gamma(\s_{f'}\times \ast)}R_{f'}(z',x')\circ (F\otimes
\Id_{\I})(z,z')\\
&= \delta_{x'_{z'},x'}\id_{f(x'_{z'})}\circ (F(x_z,x'_{z'})\otimes
\Id_{\I}(\ast,\ast))\\
&= \delta_{x'_{z'},x'}\id_{f(x'_{z'})}\circ F(x_z,x'_{z'}) \\
&= F(x_z,x')\,.
\end{split}
\end{equation}

The proof for $L$ is analogous.
\end{proof}

\begin{lem}
The morphisms $A$, $R$ and $L$ satisfy the Triangular Axiom.
\end{lem}

\begin{proof} Set $P(v,w)=((\Id_f\otimes L_g)\circ A_{f,\I,g})(v,w)$.
Then
\begin{equation}
\begin{split}
P(v,w)&=\sum_{u\in S_{f\otimes
(\I\otimes g)}}(\Id_f\otimes L_g)(u,w)\circ A_{f,\I,g}(v,u)\\
&=(\delta_{x_u,x_w}\id_{f(x_u)}\otimes
\delta_{y_u,y_w}\id_{g(y_u)})\circ
\delta^{v,u}_{x;y}\id_{f(x_v)\otimes \I(\ast)\otimes g(y_v)}\\
&=\delta_{x_v,x_w} \id_{f(x_v)}\otimes
\delta_{y_v,y_w}\id_{g(y_v)}\\
&=(R_f\otimes \Id_g)(v,w).
\end{split}
\end{equation}
 So $(\Id_f\otimes L_g)\circ A_{f,\I,g}= R_f\otimes \Id_g$.
\end{proof}

These four lemmas finish the proof of \ref{teorema: construccion rcs}

\begin{prop}\label{funtor inclusion}
The category $\rc^{\textsf{S}_0}$ has a full subcategory, which is
tensor equivalent to $\rc$.
\end{prop}

\begin{proof}  Recall that a tensor
functor is a triple $(F,\varphi_0,\varphi_2)$, where $F$ is a
functor, $\varphi_0$ is an isomorphism from $\I$ to $F(\I)$, and
$\varphi_2(U,V):F(U)\otimes F(V)\en F(U\otimes V)$ is a family of
natural isomorphisms compatible with the associative constraint and
the left and right units (see \cite[p.287]{Ka95}). Define a functor
$J:\rc \en \rc^{\textsf{S}_0}$, by choosing for any object $V$ in
$\rc$ any point $x_V\in \s_0$ and a function $f_V:\{x_V\}\en
\Obj(\rc)$, given by $f_V(x_V)=V$. Then we define $J(V)=f_V$. To any
morphism $\alpha:V\en W$ we assign the function
$F_{\alpha}(x_V,x_W)=\alpha:f_V(x_V)=V\en f_W(x_W)=W$ and then
define $J(\alpha)=F_{\alpha}$. For the unit object $\I$ of $\rc$ we
choose the fixed point $\ast$ as before, so that $J(\I)=\I\in
\Obj(\rcs)$. For $U,V$ objects of $\rc$, define
$\varphi_2(U,V):J(U)\otimes J(V)=f_U\otimes f_V \en J(U\otimes
V)=f_{U\otimes V}$, as follows. If $\gamma^{-1}(\{x_U\}\times
\{x_V\})=\{x'_{U,V}\}$, then $(f_U\otimes f_V)(x'_{U,V})=U\otimes V$
and $f_{U\otimes V}(x_{U\otimes V})=U\otimes V$, then take
$\varphi_2(U,V)(x'_{U,V},(x_{U\otimes V}))=\id_{U\otimes V}$. The
morphisms $\varphi_0$ and $\varphi_2$ are identities, so that the
functor $J$ is strict, and it is straightforward to prove that they
satisfy the required compatibility conditions.
\end{proof}

\subsection{Extending the braiding and the twist}

Let us now assume that the category $\rc$ is braided with braiding
$\ce$. For $v\in \gamma(\s_f\times \s_g)$ and $w\in
\gamma(\s_g\times \s_f)$, define $\Ce_{f,g}(v,w)$ by
\begin{equation}
\begin{split}
\Ce_{f,g}(v,w)=\delta^{v,w}_{x;y}\ce_{f(x_v),g(y_v)}:(f\otimes
g)(v)=f(x_v)\otimes g(y_v) &\en g(y_w)\otimes f(x_w) \\
                           &=(g\otimes f)(w)\,.
\end{split}
\end{equation}
It is clear that $\Ce_{f,g}$ is invertible with inverse given by
$\Ce_{f,g}^{-1}(w,v)=\delta^{w,v}_{x;y}\ce^{-1}_{f(x_w),g(y_w)}$.

\begin{prop}
The family $\Ce$ of isomorphisms $\Ce_{f,g}$ is a braiding in the category
$\mathcal{C}^{\textsf{S}_0}$.
\end{prop}

\begin{proof}
We have to prove that $\Ce$ is natural and satisfies the
Hexagonal Axiom.
 For $F:f\en f'$ and $G:g\en g'$ we have, on the one hand
\begin{equation}
\begin{split}
((G\otimes F)\circ \Ce_{f,g})(v,w') & =\sum_{w\in
\gamma(\s_g\otimes
\s_f)}(G\otimes F)(w,w')\circ \Ce_{f,g}(v,w)\\
&=\sum_{w\in \gamma(\s_g\otimes \s_f)}(G(y_w,y_{w'})\otimes
F(x_w,x_{w'}))\circ \delta^{v,w}_{x;y}\ce_{f(x_v),g(y_v)}\\
&=(G(y_v,y_{w'})\otimes F(x_v,x_{w'}))\circ \ce_{f(x_v),g(y_v)}.
\end{split}
\end{equation}
On the other hand,
\begin{equation}
\begin{split}
\Ce_{f',g'}\circ (F\otimes G)(v,w')&=\sum_{v'\in
\gamma(\s_{f'}\times \s_{g'})} \Ce_{f',g'}(v',w')\circ (F\otimes
G)(v,v')\\
&=\sum_{v'\in \gamma(\s_{f'}\times \s_{g'})}
\delta^{v',w'}_{x;y}\ce_{f'(x'_{v'}),g'(y'_{v'})}\circ
F(x_v,x'_{v'})\otimes G(y_v,y_{v'})\\
&=\ce_{f'(x'_{w'}),g'(y'_{w'})}\circ (F(x_v,x'_{w'})\otimes
G(y_v,y_{w'})).
\end{split}
\end{equation}
Both sums are equal since $\ce$ is a braiding in $\rc$ and
therefore it is natural. Thus $\Ce$ is natural. We now show the
commutativity of one of the diagrams of the Hexagonal Axiom. Put
$M(w,w')=(A_{g,h,f}\circ \Ce_{f,g\otimes h}\circ
A_{f,g,h})(w,w')$. Then
\begin{equation}
\begin{split}
M(w,w')&=\sum_{\substack{u\in \s_{(f\otimes g)\otimes h}\\ v\in
\s_f\otimes (g\otimes h)}}A_{f,g,h}(u,w')\circ \Ce_{f,g\otimes
h}(v,u)\circ A_{f,g,h}(w,v)\\
&=\sum_{\substack{u\in \s_{(f\otimes g)\otimes h}\\ v\in
\s_f\otimes (g\otimes h)}}\delta^{u,w'}_{x;y;z}
\id_{f(x_{u})\otimes g(y_{u})\otimes h(z_u)}\circ
\delta^{v,u}_{x;y;z}\ce_{f(x_v),g(y_{v})\otimes h(z_{v})}\\
&\quad \circ \delta^{w,v}_{x;y;z} \id_{f(x_{w})\otimes
g(y_{w})\otimes h(z_w)}\\
&=\delta^{w,w'}_{x;y;z}\ce_{f(x_w),g(y_{w})\otimes h(z_{w})}.
\end{split}
\end{equation}
Set $N(w,w')=((\Id_g\otimes \Ce_{f,h})\circ A_{g,f,h}\circ
(\Ce_{f,g}\otimes \Id_h))(w,w')$. Then
\begin{equation}
\begin{split}
N(w,w')&=\sum_{\substack{u\in \s_{g\otimes (f\otimes h)}\\ v\in
\s_{(g\otimes f)\otimes h}}}(\Id_g\otimes \Ce_{f,h})(u,w')\circ
A_{g,f,h}(v,u)\circ (\Ce_{f,g}\otimes \Id_h)(w,v)\\
&=\sum_{\substack{u\in \s_{g\otimes (f\otimes h)}\\ v\in
\s_{(g\otimes f)\otimes
h}}}(\delta_{y_u,y_{w'}}\id_{g(y_u)}\otimes
\delta^{u,{w'}}_{x;z}\ce_{f(x_{u}),h(z_{{u}})})\circ
\delta^{v,u}_{x;y;z}\id_{g(y_v)\otimes f(x_v)\otimes
h(z_v)}\\
&\quad \circ (\delta^{w,v}_{x;y}\ce_{f(x_{w}),g(y_{w})}\otimes
\delta_{z_w,z_v}\id_{h(z_w)})\\
&=\delta^{w,w'}_{x;y;z}(\id_{g(y_w)}\otimes
\ce_{f(x_{w}),h(z_{{w}})})\circ (\ce_{f(x_{w}),g(y_{w})}\otimes
\id_{h(z_w)}).
\end{split}
\end{equation}
Again, since $\ce$ is a strict braiding in $\rc$, we have the
equality $M(w,w')=N(w,w')$. The commutativity of the other hexagon
is proved analogously.
\end{proof}
 In the same way, if the category $\rc$ has a twist, then we can easily prove the following
 assertion.

 \begin{prop}
 Let $\theta$ be a twist for the the category $\rc$. Then the category
 $\rc^{\textsf{S}_0}$ has a twist $\Theta_f:f\en f$ given by
 \begin{equation}
 \Theta_f(x,y)=\delta_{x,y}\theta_{f(x)}:f(x)\en f(y)
\end{equation}
for any object $f$ in $\rc^{\textsf{S}_0}$. $\hbx$
 \end{prop}

\

 However, it is not possible to extend a duality from $\rc$ to
 $\rc^{\textsf{S}_0}$. Although we have for any $f:\s_f\en \Obj(\rc)$ a canonical
 candidate for $f^{\ast}:\s_f\en \Obj(\rc)$, namely the function $f^{\ast}$ defined by
 $f^{\ast}(x)=(f(x))^{\ast}$ as well as a canonical candidate for the
 evaluation $D_f:f^{\ast}\otimes f\en \I$, given by
  $D_f(v,\{\ast\})=
 \delta_{x^{\ast}_v,x_v}\de_{f(x_v)}:f^{\ast}(x^{\ast}_v)\otimes f(x_v)\en
 \I(\ast)=\I$, where $\gamma^{-1}(v)=(x^{\ast}_v,x_v)\in \s_f\times
 \s_f$, this is not the case for the coevaluation. Indeed, the
 canonical extension $B_f:\I\en f\otimes f^{\ast}$ given by
 $B_f(\ast,v)=\delta_{x^{\ast}_v,x_v}
 \be_{f(x_v)}:\I\en f(x_v)\otimes f^{\ast}(x^{\ast}_{v})$ is not a morphism in
 $\rc^{\textsf{S}_0}$ if $\s_f$ is infinite, since condition (ii) of
 page \pageref{def. de morfismo en cs} does not hold.

 Nevertheless, if we consider the full subcategory
 $\rc^{\textsf{S}_0}_\sharp$ which as objects has
 functions $f$ with finite domain $\s_f$, then it is possible to
 extend the duality according to the given formulas.
 It is easy to see that the inclusion functor $J:\rc\en \rc^{\textsf{S}_0}$ factors through
 $\rc^{\textsf{S}_0}_\sharp$, i.e.,
 \begin{equation}
 \xymatrix@1{J:\rc\,\, \ar @{^{(}->}[r]& \,\rc^{\textsf{S}_0}_\sharp\,\, \ar@ {^{(}->}[r]&
 \,\rc^{\textsf{S}_0}}
\end{equation}
The following assertion is also easy to prove.

\begin{prop}
 If the category $\rc$ is a ribbon category, then the extended structure in
 $\rc^{\textsf{S}_0}_\sharp$ is pivotal braided (but nonstrict in
 general, so it is not ribbon).

$\hbx$
 \end{prop}

\begin{Nota}
In order to simplify the next computations, we shall adopt the
following notation. Let $\mathcal{A}$ be the set of isomorphisms of
$\rcs$ generated by the set $(\Id_\chi, A^{\pm
1}_{\kappa,\lambda,\mu},R_\zeta, L_\varsigma)$ under tensor products
and compositions, where $\chi$, $\kappa$, $\lambda$, $\mu$, $\zeta$,
and $\varsigma$ are any objects in $\rcs$. In other words,
$\mathcal{A}$ is the set of isomorphisms that relate different
objects by associativity and units. If $F$ and $G$ are morphisms in
$\rcs$, we shall write $F\doteq G$ if $G=X\circ F\circ Y$, where $X$
and $Y$ are elements of $\mathcal{A}$. For example, $F\doteq G$ if
the following diagram commutes.
\begin{equation*}
\xymatrix{(((f_1\otimes f_2)\otimes f_3)\otimes f_4)\otimes
f_5\ar[r]^-{F}\ar[d]_{A_{f_1\otimes f_2,f_3,f_4}\otimes
\id_{f_5}}& (g_1\otimes
g_2)\otimes g_3\ar[ddd]^-{A_{g_1,g_2,g_3}}\\
((f_1\otimes f_2)\otimes (f_3\otimes f_4))\otimes
f_5\ar[d]_-{A_{f_1\otimes f_2,f_3\otimes f_4,f_5}}&\\
(f_1\otimes f_2)\otimes ((f_3\otimes f_4)\otimes
f_5)\ar[d]_-{\Id_{f_1\otimes f_2}\otimes A^{-1}_{f_3,f_4,f_5}}&\\
(f_1\otimes f_2)\otimes (f_3\otimes (f_4\otimes f_5))\ar[r]^-{G }&
g_1\otimes (g_2\otimes g_3)}
\end{equation*}
The relation $\doteq$ is an equivalence relation in the set of
morphisms of $\rcs$ which is compatible with composition and tensor
product in the sense that if $F\doteq G$ and $F'\doteq G'$ then
$F'\circ F\doteq G'\circ G$, if the compositions are defined, and
$F\otimes F'\doteq G\otimes G'$. Indeed, for the composition,
suppose $A\circ F\circ B= G$ and that $C\circ F'\circ D= G'$, for
elements $A$, $B$, $C$, and $D$ in $\mathcal{A}$. Then $ G'\circ
G=C\circ F'\circ D\circ A\circ F\circ B$. The morphism $D\circ A$ is
an endomorphism of the domain $s(F')$ of $F'$ which is equal to the
codomain $t(F)$ of $F$ and is an element of $\mathcal{A}$. Mac
Lane's coherence theorem states that this element has to be the
identity morphism $\Id_{s(F')}$. Hence $ G'\circ G=C\circ F'\circ
F\circ B$. The tensor part follows from the identity $(A\circ F\circ
B)\otimes (C\circ F'\circ D)=(A\otimes C)\circ (F\otimes F')\circ
(B\otimes D)$. In what follows we shall use this notation without
further comments.
\end{Nota}

\section{Bialgebras in $\rc^{\textsf{S}_0}$}\label{section2}
Let $\ve$ be a monoidal category. We say that an object $A$ of $\ve$
is an {\em algebra in} $\ve$, if there exist morphisms $\mu:A\otimes
A\en A$ and $\eta:\I\en A$ such that
 \begin{gather}
\mu(\mu\otimes \id_A)\doteq\mu(\id_A\otimes \mu)\,,\\
\mu(\eta\otimes \id_{A})\doteq \id_{A}\doteq \mu(\id_{A}\otimes
\eta).
\end{gather}
Dually, we say that $C$ is a \emph{coalgebra in} $\ve$, if there exist
morphisms $\Delta:C\en C\otimes C$ and $\varepsilon:C\en \I$ such
that
\begin{gather}
({\Delta}\otimes \id_{{C}}){\Delta}\doteq(\id_{{C}}\otimes
{\Delta}){\Delta}\,,\\
({\varepsilon}\otimes \id_{{C}}){\Delta}\doteq \id_{{C}}\doteq
(\id_{{C}}\otimes {\varepsilon}){\Delta}.
\end{gather}
If $H$ is an algebra, then the product in $H\otimes H$ is defined
by the following composite
\begin{equation}\label{producto en HtH}
\xymatrix@C=1.6cm{\widehat{\mu}:(H\otimes H)\otimes (H\otimes
H)\ar[r]^-{A^{-1}_{H\otimes H,H,H}}& ((H\otimes H)\otimes
H)\otimes H \ar[r]^-{A_{H,H,H}\otimes \id_{H}} & {}\\
{(H\otimes (H\otimes H))\otimes H}\ar[r]^-{\id_H\otimes
\ce_{H,H}\otimes \id_{H}}& (H\otimes (H\otimes H))\otimes H
\ar[r]^-{A^{-1}_{H,H,H}\otimes \id_H}&{}\\
((H\otimes H)\otimes H)\otimes H\ar[r]^-{A_{(H\otimes H), H,H}}&
(H\otimes H)\otimes (H\otimes H)\ar[r]^-{\mu\otimes
\mu}&{}\\
{}\ar[r]&H\otimes H\,.}
\end{equation}We say that $H$ is a {\em bialgebra in} $\ve$, if
$\widehat{\mu}(\Delta\otimes \Delta)\doteq \Delta \mu$ and
$\varepsilon\mu=\varepsilon\otimes \varepsilon$. \\
If $A$ is an algebra, an object $V$ is an $A$-module, if there
exists a morphism $T:A\otimes V\en V$, such that $T(\mu\otimes
\id_V)\doteq T(\id_A\otimes T)$ and $T(\eta\otimes \id_V)\doteq \id_V$.\\
 Note that if the category is
strict monoidal, the latter are the concepts of algebra, coalgebra,
bialgebra and module in strict braided monoidal categories.

We are going to find  bialgebras in $\rc^{\textsf{S}_0}$, when $\rc$
is a braided strict monoidal category with left duality.

Let $h:\s_h\en \Obj(\rc)$ be an injective function such that
$h(\s_h)\subset \Obj(\rc)$ is closed under $\otimes$, that is, for
any pair $(x,y)\in \s_h\times \s_h$, there exists a unique $z\in
\s_h$ such that $h(x)\otimes h(y)=h(z)$ and suppose $\I \in
h(\s_h)$. For example, we can take a set $\s_0$ with the same
cardinality as $\Obj(\rc)$ and $h:\s_0\en \Obj(\rc)$ to be any
bijection, if $\Obj(\rc)$ is an infinite set.

Set $\Delta_h=\{(x,x)\mid x\in \s_h\}\subset \s_h\times \s_h$ and
let $\h$ be the object defined by the composite
\begin{equation}
\xymatrix@1{\h:\gamma(\Delta_h)\ar[r]^-{\gamma^{-1}}&
\Delta_h\ar[r]^-{h^{\ast}\times h}& \Obj(\rc)\times
\Obj(\rc)\ar[r]^-{\otimes}&\Obj(\rc)}
\end{equation}
where $h^{\ast}(x):=(h(x))^{\ast}$. That is, $\h$ is defined by
the relation $\h(\gamma(x,x))=h^{\ast}(x)\otimes h(x)$, for
$\gamma(x,x)\in \s_{\h}=\gamma(\Delta_h)$.

 The main theorem in this section is the following.

\begin{thm}\label{principal} The object $\h$ is a bialgebra in
$\rcs$ and the objects of $\rc$, considered as a subcategory of
$\rcs$, are $\h$-modules.
\end{thm}

 To prove it, we shall establish two previous lemmas.  Let $\chi:\s_h\times \s_h\en
 \s_h$ be the function defined by the relation $h(\chi(x,y))=h(x)\otimes
 h(y)$.

  \begin{lem}\label{lema1}
  The function $\chi$ satisfies $\chi(\chi(x,y),z)=\chi(x,\chi(y,z))$.
  \end{lem}

\begin{proof} $$h(\chi(\chi(x,y),z))=h(\chi(x,y))\otimes h(z)=h(x)\otimes
h(y)\otimes h(z)=$$
$$=h(x)\otimes h(\chi(y,z))=h(\chi(x,\chi(y,z)))\,.$$
Thus $\chi(\chi(x,y),z)=\chi(x,\chi(y,z))$.
\end{proof}

We now consider, in a more general situation, a monoidal category
with left duality \ve. In the following lemma we use letters
$x,y,z,...$ to denote objects of \ve.
  Let $x,y$ be objects of $\ve$. Recall that there exists an isomorphism
  $\gamma_{x,y}:y^\ast\otimes x^\ast\en
(x\otimes y)^\ast$ given by
\begin{equation}\label{iso.tens.dual.1}
\gamma_{x,y}=(\de_y\otimes \id_{(x\otimes y)^\ast})(\id_{y^\ast}
\otimes \de_x\otimes \id_{y\otimes(x\otimes
y)^\ast})(\id_{y^\ast\otimes x^\ast}\otimes
 \be_{x\otimes y})\,.
 \end{equation}
 Now define the isomorphism $\Gamma_{x,y}:y^\ast\otimes
y\otimes x^\ast \otimes x\en (x\otimes y)^\ast\otimes (x\otimes
y)$ by the composite
$$ \xymatrix@C=0.5cm{ \Gamma_{x,y}: y^\ast\otimes y\otimes x^\ast
\otimes x \ar[rrr]^-{\id_{y^\ast}\otimes \ce_{y,x^\ast}\otimes
\id_x}&&& y^\ast\otimes x^\ast\otimes y\otimes x
\ar[rr]^-{\gamma_{x,y}\otimes \ce_{y,x}}&& (x\otimes
y)^\ast\otimes (x\otimes y)\,.}$$

\begin{lem}  \label{lema2}The isomorphisms $\Gamma_{x,y}$ satisfy the relation $$\Gamma_{x,y\otimes z}
(\Gamma_{y,z}\otimes \id_{x^\ast\otimes x})=\Gamma_{x\otimes y,z}
(\id_{z^\ast\otimes z}\otimes \Gamma_{x,y}).$$
That is, if $x,y$ and $z$ are objects of \ve\, then the following
diagram commutes
$$ \xymatrix{ z^\ast\otimes z\otimes y^\ast\otimes y\otimes
x^\ast\otimes x \ar[rr]^-{\Gamma_{y,z}\otimes \id_{x^\ast\otimes
x}} \ar[d]_{\id_{z^\ast\otimes z}\otimes \Gamma_{x,y}} & &
(y\otimes x)^\ast\otimes (y\otimes x)\otimes x^\ast\otimes x
\ar[d]^{\Gamma_{x,y\otimes z}} \\
 z^\ast\otimes z\otimes
(x\otimes y)^\ast\otimes (x\otimes y) \ar[rr]^{\Gamma_{x\otimes
y,z}} & & (x\otimes y\otimes
 z)^\ast\otimes (x\otimes y\otimes z)\,. }
$$
\end{lem}

 \begin{proof} We prove it by using graphical calculus. In Figure \ref{fig4.1} the morphism
 $\Gamma_{x,y}$ is represented. Figure \ref{fig4.2} (at the end of the
 paper) proves the
 Lemma. The upper left and bottom right diagrams represent
 the morphisms $\Gamma_{x,y\otimes
z} (\Gamma_{y,z}\otimes \id_{x^\ast\otimes x})$ and
$\Gamma_{x\otimes y,z}
(\id_{z^\ast\otimes z}\otimes \Gamma_{x,y})$,  respectively.

\begin{figure}[ht]
\begin{center}
\includegraphics[width=9.5cm] {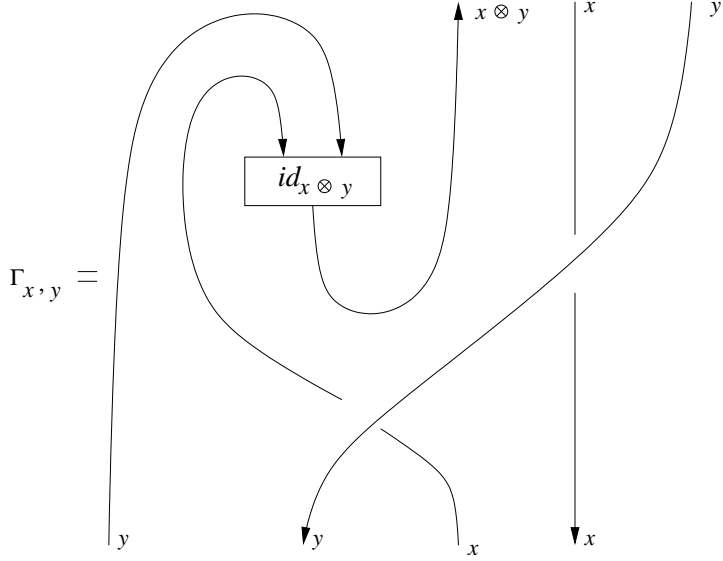}
\caption{The morphism $\Gamma_{x,y}$ } \label{fig4.1}
\end{center}
\end{figure}

\begin{figure}[ht]
\begin{center}
\includegraphics[width=16.5cm] {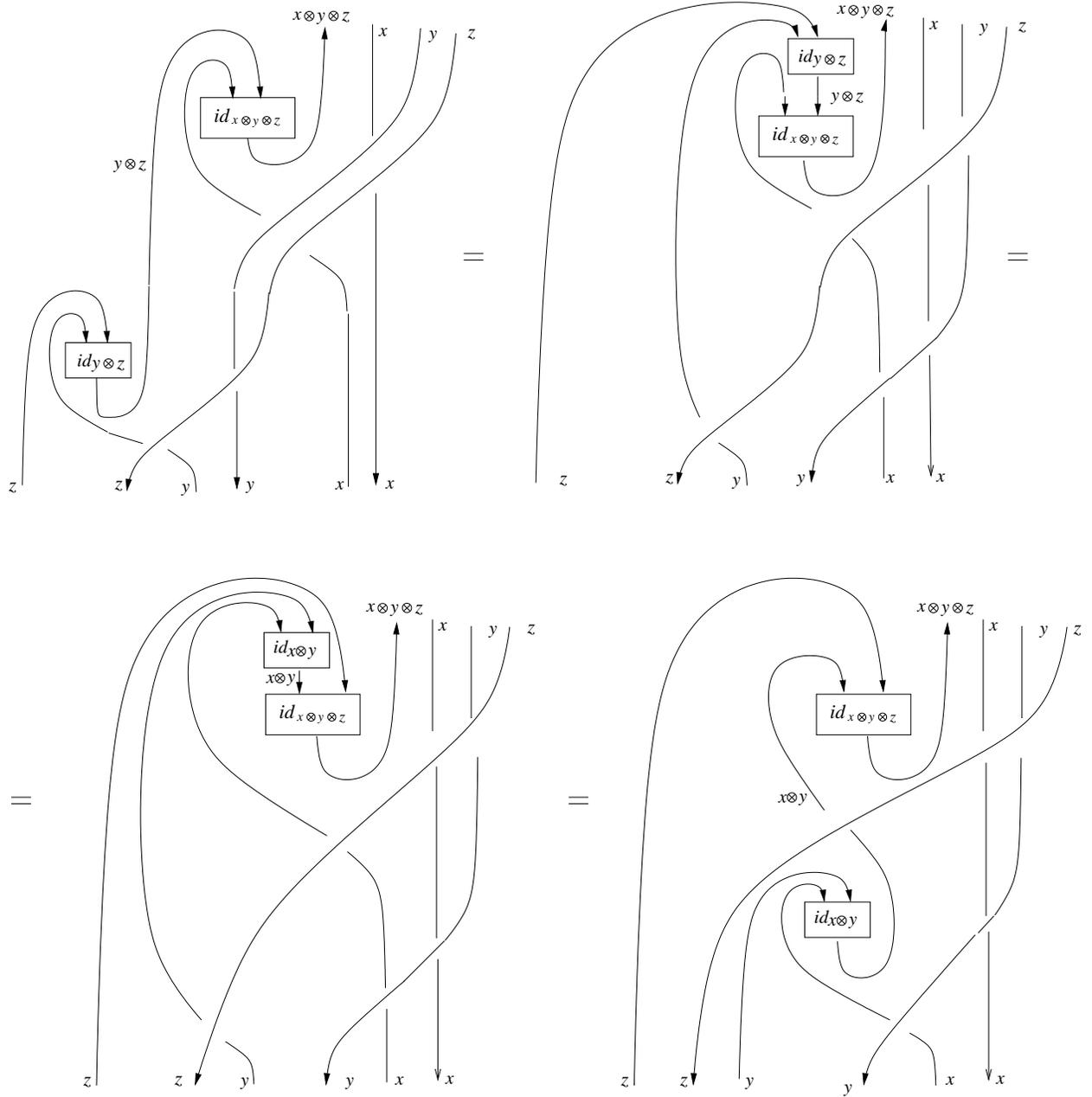}
\caption{$\Gamma_{x,y\otimes z} (\Gamma_{y,z}\otimes
\id_{x^\ast\otimes x})=\Gamma_{x\otimes y,z} (\id_{z^\ast\otimes
z}\otimes \Gamma_{x,y})$  } \label{fig4.2}
\end{center}
\end{figure}
\end{proof}

\noindent
 {\em Proof} of Theorem \ref{principal}. Define $\mu:\h\otimes \h\en
 \h$ by
 \begin{equation}
\begin{matrix}
 \mu(v,\gamma(z,z)):(\h\otimes \h)(v)=h^{\ast}(x_v)\otimes h(x_v)\otimes h^{\ast}(y_v)\otimes
 h(y_v)  \\
 \xymatrix{
{}\ar[rrr]^-{\delta_{z,\chi(y_v,x_v)\Gamma_{h(y_v),h(x_v)}}} & & &
 \h(\gamma(z,z))=h^{\ast}(z)\otimes h(z)}
\end{matrix}
\end{equation}
Since there is a unique $x_0\in \s_h$ such that $h(x_0)=\I\in
\Obj(\rc)$, we can define $\eta:\I\en \h$ by
$$\eta(\ast,\gamma(y,y))=\delta_{x_0,y}\id_{\I}:\I=h^{\ast}(x_0)\otimes
h(x_0)\en h^{\ast}(y)\otimes h(y)\,.$$
We have to prove now that
$\mu(\mu\otimes \Id_{\h})\doteq\mu(\Id_{\h}\otimes \mu)$ and
$\mu(\eta\otimes \Id_{\h})=\Id_{\h}= \mu(\Id_{\h}\otimes \eta)$.\\
Set $S=\mu(\mu\otimes \Id_{\h})(w,\gamma(t,t))$ and
$R=\mu(\Id_{\h}\otimes \mu)(w',\gamma(t,t))$. We have on the one hand
\begin{equation}
\begin{split}
S&=\sum_{v\in \s_{\h\otimes
\h}}\mu(v,\gamma(t,t))\circ (\mu\otimes \Id_{\h})(w,v)\\
&=\sum_{v\in \s_{\h\otimes
\h}}\delta_{t,\chi(y_v,x_v)}\Gamma_{h(y_v),h(x_v)}\circ
(\delta_{x_v,\chi(y_w,x_w)}\Gamma_{h(y_w),h(x_w)}\otimes
\delta_{z_w,y_v}\id_{h^{\ast}(z_w)\otimes h(z_w)})\\
 &=\delta_{t,\chi(z_w,\chi(y_w,x_w))}\Gamma_{h(z_w),h(\chi(y_w,x_w))}\circ
(\Gamma_{h(y_w),h(x_w)}\otimes \id_{h^{\ast}(z_w)\otimes
 h(z_w)})\\
 &=\delta_{t,\chi(z_w,\chi(y_w,x_w))}\Gamma_{h(z_w),h(y_w)\otimes
  h(x_w)}\circ (\Gamma_{h(y_w),h(x_w)}\otimes
\id_{h^{\ast}(z_w)\otimes
 h(z_w)})
\end{split}
\end{equation}
 On the other hand we have
\begin{equation}
\begin{split}
R=&\sum_{v\in \s_{\h\otimes \h}}\mu(v,\gamma(t,t))\circ
(\Id_{\h}\otimes \mu)(w',v)\\
&=\sum_{v\in \s_{\h\otimes
\h}}\delta_{t,\chi(y_v,x_v)}\Gamma_{h(y_v),h(x_v)}\circ
(\delta_{x_{w'},x_v}\id_{h^{\ast}(x_{w'})\otimes h(x_{w'})}\otimes
\delta_{y_v,\chi(z_{w'},y_{w'})}\Gamma_{h(z_{w'}),h(y_{w'})})\\
&=\delta_{t,\chi(\chi(z_{w'},y_{w'}),x_{w'})}\Gamma_{h(\chi(z_{w'},y_{w'})),h(x_v)}\circ
(\id_{h^{\ast}(x_{w'})\otimes h(x_{w'})}\otimes \Gamma_{h(z_{w'}),h(y_{w'})}))\\
&=\delta_{t,\chi(\chi(z_{w'},y_{w'}),x_{w'})}\Gamma_{h(z_{w'})\otimes
h(y_{w'}),h(x_v)}\circ (\id_{h^{\ast}(x_{w'})\otimes
h(x_{w'})}\otimes \Gamma_{h(z_{w'}),h(y_{w'})}))
\end{split}
\end{equation}
According to Lemma \ref{lema1}, we have
$\chi(\chi(z_{w'},y_{w'}),x_{w'})=\chi(z_{w'},\chi(y_{w'},x_{w'}))$.
From this and Lemma \ref{lema2}, it is easy to see that
$R\circ A_{\h,\h,\h}=S$ so $R\doteq S$.

We shall prove now that $\mu(\eta\otimes \Id_{\h})\doteq \Id_{\h}$.
Set $J=\mu(\eta\otimes \Id_{\h})(u,\gamma(z,z))$. From
$h(\chi(x_u,x_0))=h(x_u)\otimes h(x_0)=h(x_u)\otimes \I=h(x_u)$ we
deduce that $\chi(x_u,x_0)=x_u$ and since
$\Gamma_{a,\I}=\id_{h^{\ast}(a)\otimes h(a)}$ for any object $a$ of
$\rc$, we have
\begin{equation}
\begin{split}
J&=\sum_{v\in \s_{\h\otimes
\h}}\mu(v,\gamma(z,z))\circ (\eta\otimes \Id_{\h})(u,v)\\
&=\sum_{v\in \s_{\h\otimes
\h}}\delta_{z,\chi(y_v,x_v)}\Gamma_{h(y_v),h(x_v)}\circ
(\eta(\ast,\gamma(x_v,x_v))\otimes
\Id_{\h}(\gamma(x_u,x_u),\gamma(y_v,y_v)))\\
&=\sum_{v\in \s_{\h\otimes
\h}}\delta_{z,\chi(y_v,x_v)}\Gamma_{h(y_v),h(x_v)}\circ
(\delta_{x_0,x_v}\id_{\I}\otimes
\delta_{x_u,y_v}\id_{h^{\ast}(x_u)\otimes h(x_u)})\\
&=\delta_{z,\chi(x_u,x_0)}\Gamma_{h(x_u),h(x_0)}\\
&=\delta_{z,x_u}\Gamma_{h(x_u),\I}\\
&=\delta_{z,x_u}\id_{h^{\ast}(x_u)\otimes h(x_u)}\\
&=\Id_{\h}(u,\gamma(z,z))
\end{split}
\end{equation}
The relation $\mu(\Id_{\h}\otimes \eta)\doteq \Id_{\h}$ is proved in
a similar way.

We have thus shown that $(\h,\mu,\eta)$ is an algebra in
 $\rcs$.
 Define now $\Delta:\h\en \h\otimes \h$ to be the function
 $$\Delta(\gamma(x,x),v):h^{\ast}(x)\otimes h(x)\en
 h^{\ast}(y_v)\otimes h(y_v)\otimes h^{\ast}(z_v)\otimes h(z_v)$$
 given by the following composite
\begin{equation*}
\begin{matrix}
\xymatrix@1{{\delta_{x,y_v}\delta_{x,z_v}\id_{h^{\ast}(x)}\otimes
\be_{h(x)}\otimes \id_{h(x)}}:h^{\ast}(x)\otimes h(x)\ar[r] & {}} \\
\xymatrix@1{{}\ar[r] & h^{\ast}(y_v)\otimes h(y_v)\otimes
  h^{\ast}(z_v)\otimes h(z_v)}
\end{matrix}
\end{equation*}
and define $\varepsilon:\h\en \I$ as the function given by
$$\varepsilon(\gamma(x,x),\ast)=\de_{h(x)}:h^{\ast}(x)\otimes
h(x)\en \I\,.$$
We are going to prove that $(\Id_{\h}\otimes
\Delta)\Delta\doteq (\Delta\otimes \Id_{\h}\Delta)$ and
$(\varepsilon\otimes \Id_{\h})\Delta=\Id_{\h}=(\Id_{\h}\otimes
\varepsilon)$. Set $L=(\Id_{\h}\otimes
\Delta)\Delta(\gamma(t,t),w)$. Then
\begin{equation}
\begin{split}
L&=\sum_{v\in \s_{\h\otimes \h}}(\Id_{\h}\otimes
\Delta)(v,w')\circ
\Delta(\gamma(t,t),v)\\
 &=\sum_{v\in \s_{\h\otimes
\h}}(\delta_{x_v,x_w}\id_{h^{\ast}(x_v)\otimes h(x_v)}\otimes
\delta_{y_v,y_{w'}}\delta_{y_v,z_{w'}}\id_{h^{\ast}(y_v)}\otimes
\be_{h(y_v)}\otimes \id_{h(y_v)})\circ\\
&\quad ({\delta_{t,x_v}\delta_{t,y_v}\id_{h^{\ast}(t)}\otimes
\be_{h(t)}\otimes \id_{h(t)}})\\
&=\delta_{t,x_v}\delta_{t,y_v}\delta_{t,z_v}(\id_{h^{\ast}(t)\otimes
h(t)}\otimes \id_{h^{\ast}(t)}\otimes \be_{h(t)}\otimes
\id_{h(t)})\circ (\id_{h^{\ast}(t)}\otimes \be_{h(t)}\otimes
\id_{h(t)})
\end{split}
\end{equation}
Set $R=(\Delta\otimes \Id_{\h}\Delta)(\gamma(t,t),w)$. Then
\begin{equation}
\begin{split}
R&=\sum_{v\in \s_{\h}\otimes \s_{\h}}(\Delta\otimes
\Id_{\h})(v,w)\circ \Delta(\gamma(t,t),v)\\
&=\sum_{v\in \s_{\h}\otimes
\s_{\h}}\delta_{x_v,x_w}\delta_{x_v,y_w}(\id_{h^{\ast}(x_v)}\otimes
\be_{h(x_v)}\otimes \id_{h(x_v)}\otimes
\delta_{y_v,z_w}\id_{h^{\ast}(y_v)\otimes h(y_v)})\circ\\
&\quad ({\delta_{t,x_v}\delta_{t,y_v}\id_{h^{\ast}(t)}\otimes
\be_{h(t)}\otimes \id_{h(t)}})\\
&=\delta_{t,x_v}\delta_{t,y_v}\delta_{t,z_v}(\id_{h^{\ast}(t)}\otimes
\be_{h(t)}\otimes \id_{h(t)}\otimes \id_{h^{\ast}(t)\otimes
h(t)})\circ (\id_{h^{\ast}(t)}\otimes \be_{h(t)}\otimes
\id_{h(t)})
\end{split}
\end{equation}
Taking $x=h(t)$, Figure \ref{fig4.3} (at the end of the paper) shows that $R$ and $L$
are equal up to associativity, that is
$A_{\h,\h,\h}(w,w')\circ R=L$. Thus  $L\doteq R$.
\begin{figure}[ht]
\begin{center}
\includegraphics[width=15.5cm] {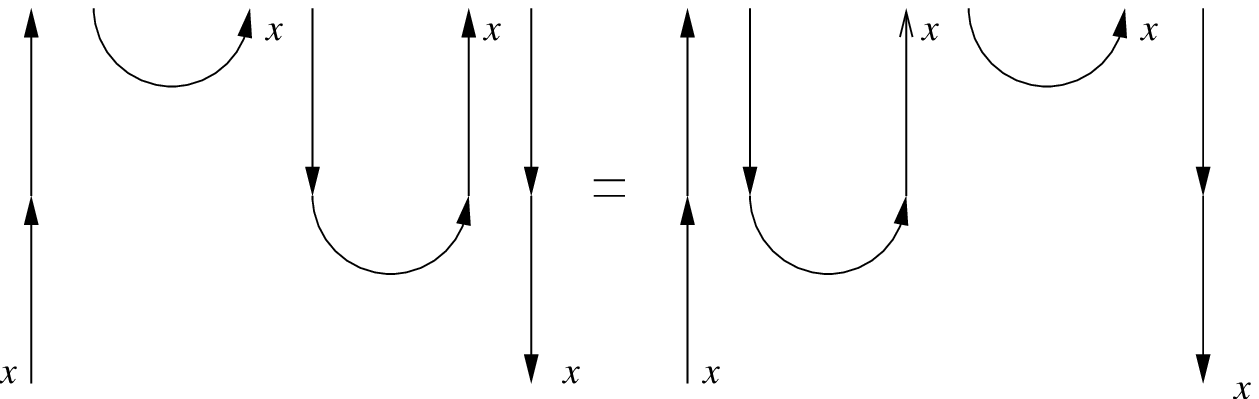}
\caption{$(\id_{h_i^\ast}\otimes \be_{h_i}\otimes \id_{h_i}\otimes
\id_{h_i^\ast\otimes h_i})(\id_{h_i^\ast}\otimes \be_{h_i}\otimes
\id_{h_i})=(\id_{h_i^\ast\otimes h_i}\otimes \id_{h_i^\ast}\otimes
\be_{h_i}\otimes \id_{h_i})(\id_{h_i^\ast}\otimes \be_{h_i}\otimes
\id_{h_i})$}  \label{fig4.3}
\end{center}
\end{figure}

Next we prove $(\varepsilon\otimes
\Id_{\h})\Delta=\Id_{\h}$. Set $J=(\varepsilon\otimes
\Id_{\h})\Delta(\gamma(x,x),\gamma(y,y))$. Then
\begin{equation}
\begin{split}
J&=\sum_{x\in \s_{\h}\otimes \s_{\h}}(\varepsilon\otimes
(\Id_{\h}))(v,\gamma(y,y))\circ \Delta(\gamma(x,x),v)\\
&=\sum_{x\in \s_{\h}\otimes
\s_{\h}}(\varepsilon(\gamma(x_v,x_v),\ast)\otimes
\Id_{\h}(\gamma(y_v,y_v),\gamma(y,y)))\circ\\
&\quad (\delta_{x,x_v}\delta_{x,y_v}\id_{h^{\ast}(x)}\otimes
\be_{h(x)}\otimes \id_{h(x)})\\
&=(\de_{h(x_v)}\otimes \delta_{y_v,y}\id_{h^{\ast}(y_v)\otimes
h(y_v)})\circ
(\delta_{x,x_v}\delta_{x,y_v}\id_{h^{\ast}(x)}\otimes
\be_{h(x)}\otimes \id_{h(x)})\\
&=\delta_{x,y}(\de_{h(x)}\otimes \id_{h^{\ast}(x)\otimes
h(x)})\circ (\id_{h^{\ast}(x)}\otimes \be_{h(x)}\otimes
\id_{h(x)})
\end{split}
\end{equation}
From the definition of left duality we get $(\de_{h(x)}\otimes
\id_{h^{\ast}(x)})(\id_{h^{\ast}(x)}\otimes
\be_{h(x)})=\id_{h^{\ast}(x)}$, so
$J=\delta_{x,y}\id_{h^{\ast}(x)\otimes
h(x)}=\Id_{\h}(\gamma(x,x),\gamma(y,y))$.

The relation $\Id_{\h}=(\Id_{\h}\otimes \varepsilon)$ is proved in
a similar way, and with this we have showed
$(\h,\Delta,\varepsilon)$
is a coalgebra in $\rcs$.

It is enough to prove now that $\Delta$ and $\varepsilon$ are
algebra morphisms. For ${\Delta}$ we have to show that the diagram
$$\xymatrix{
{\h}\otimes {\h} \ar[r]^-{{\Delta}\otimes {\Delta}} \ar[d]_{{\mu}}
& ({\h}\otimes {\h})\otimes ({\h}\otimes
{\h}) \ar[d]^{{\hat{\mu}}} \\
{\h} \ar[r]^{{\Delta}} & {\h}\otimes {\h}}$$ commutes up to the
relation $\doteq$, where ${\hat{\mu}}$ is the product in
${\h}\otimes {\h}$ and it is defined, as in
 \eqref{producto en HtH}, by the composite
$$ \xymatrix @C=1.3cm{ \hat{\mu}: (\h\otimes \h)\otimes (\h\otimes
\h)\ar[r]^-{A^{-1}_{\h\otimes \h,\h,\h}}& ((\h\otimes \h)\otimes
\h)\otimes \h\ar[d]|{(\Id_{\h}\otimes
\Ce_{\h,\h}\otimes \Id_{\h})A^{-1}_{\h,\h,\h}}\\
 & (\h\otimes (\h\otimes \h))\otimes \h \ar[d]|{(\mu\otimes
  \mu)A_{\h\otimes \h,\h,\h}(A^{-1}_{\h,\h,\h}\otimes \id_{\h})}\\
& \,\,\h\otimes \h\,. }$$ The morphism $\Id_{\h}\otimes
\Ce_{\h,\h}\otimes \Id_{\h}(v,w):(\h\otimes (\h\otimes \h)\otimes
\h)(v)\en (\h\otimes (\h\otimes \h)\otimes \h)(w)$ is related to
$F_w(v,w):(\h\otimes (\h\otimes \h)\otimes \h)(v)\en (\h\otimes
\h)\otimes (\h\otimes \h)(w)$, which is represented by the following
vertical arrow
 $$\xymatrix@!C=1.7cm { h(x_v)^\ast\otimes
h(x_v)\otimes h(y_v)^\ast\otimes h(y_v)\otimes h(z_v)^\ast\otimes
h(z_v)\otimes h(t_v)^\ast\otimes
h(t_v)\ar[d]|{F_w(v,w)=\delta_{x}^{v,w}\id_{ h(x_v)^\ast\otimes
h(x_v)}\otimes \delta_{y;z}^{v,w}\ce_{ h(y_v)^\ast\otimes h(y_v),
h(z_v)^\ast\otimes h(z_v) }\otimes \delta_{t}^{v,w}\id_{
h(t_v)^\ast\otimes h(t_v)}}\\
h(x_w)^\ast\otimes h(x_w)\otimes  h(z_w)^\ast\otimes h(z_w)\otimes
h(y_w)^\ast\otimes h(y_w)\otimes h(t_w)^\ast\otimes h(t_w)}$$
since their codomains are related by associativity.

The morphism $(\mu\otimes \mu)(w,u):(\h\otimes \h)\otimes
(\h\otimes \h)\en \h\otimes \h$, is represented by the following
vertical arrow
$$\xymatrix@!C=1.0cm{h(x_w)^\ast\otimes h(x_w)\otimes
h(z_w)^\ast\otimes h(z_w)\otimes h(y_w)^\ast\otimes h(y_w)\otimes
h(t_w)^\ast\otimes h(t_w)
\ar[d]|{\delta_{x_u,\chi(z_w,x_w)}\Gamma_{h(z_w),h(x_w)}\otimes
\delta_{y_u,\chi(t_w,y_w)}\Gamma_{h(t_w),h(y_w)}} \\
 h(x_u)^\ast\otimes h(x_u)\otimes h(y_u)^\ast\otimes h(y_u)}$$
 It is not difficult to see that $\hat{\mu}\doteq \sum_w(\mu\otimes \mu)\circ F_w$
 and that this last morphism turns out to be equal to
 $$\xymatrix@!C=1.5cm{
  h(x_v)^\ast\otimes h(x_v)\otimes h(y_v)^\ast\otimes
h(y_v)\otimes h(z_v)^\ast\otimes h(z_v)\otimes  h(t_v)^\ast\otimes
h(t_v) \ar[d]|{ \delta_{x_u,\chi(z_v,x_v)} \delta_{y_u,\chi(
t_v,y_v)}(\Gamma_{h(z_v),h(x_v)}\otimes
\Gamma_{h(t_v),h(y_v)})(\id_{ h(x_v)^\ast\otimes h(x_v)}\otimes
\ce_{ h(y_v)^\ast\otimes h(y_v), h(z_v)^\ast\otimes h(z_v)}\otimes
\id_{h(t_v)^\ast\otimes h(t_v)})}\\
  h(x_u)^\ast\otimes h(x_u)\otimes h(y_u)^\ast\otimes h(y_u)} $$
  Hence $\hat{\mu}(\Delta\otimes \Delta)\doteq \sum_vG_v\circ (\Delta\otimes \Delta)$,
  where $G_v$ is the last vertical arrow. But
  $(\Delta\otimes \Delta)(p,v):(\h\otimes \h)(p)\en ((\h\otimes \h)\otimes (\h\otimes
  \h))(v)$ is given by$$
(\Delta\otimes
\Delta)(p,v)=\delta_{x_p,x_v}\delta_{x_p,y_v}\delta_{y_p,z_v}\delta_{y_p,t_v}
  (\id_{h(x_p)^\ast}\otimes \be_{h(x_p)}\otimes
  \id_{h(x_p)})(\id_{h(y_p)^\ast}\otimes \be_{h(y_p)}\otimes
  \id_{h(y_p)})$$
so the sum yields
\begin{equation*}
\begin{split}
M&=\delta_{x_u,\chi(y_p,x_p)}\delta_{y_u,\chi(y_p,x_p)}(\Gamma_{y_p,x_p}\otimes
\Gamma_{y_p,x_p})(\id_{h(x_p)^\ast\otimes h(x_p)}\otimes
\ce_{h(x_p)^\ast\otimes h(x_p),h(y_p)^\ast\otimes h(y_p)}\\
&\quad \otimes \id_{h(y_p)^\ast\otimes
h(y_p)})(\id_{h(x_p)^\ast}\otimes \be_{h(x_p)}\otimes
\id_{h(x_p)})(\id_{h(y_p)^\ast}\otimes \be_{h(y_p)}\otimes
\id_{h(y_p)})
\end{split}
\end{equation*}
On the other hand, $(\mu \Delta)(p,u)$ is the sum over $v$ of the
following composite
$$\xymatrix{
 h(x_p)^\ast\otimes h(x_p)\otimes
h(y_p)^\ast\otimes h(y_p)
\ar[rr]^-{\delta_{x_v,\chi(y_p,x_p)}\Gamma_{h(y_p),h(x_p)}} & &
h(x_v)^\ast\otimes h(x_v)
\ar[d]|{\delta_{x_v,x_u}\delta_{x_v,y_u}(\id_{h(x_v)^\ast}\otimes
\be_{h(x_v)}\otimes \id_{h(x_v)})} \\
& &    h(x_u)^\ast\otimes h(x_u)\otimes h(y_u)^\ast\otimes h(y_u)}
$$
which is equal to
\begin{equation*}
(\mu\Delta)(p,u)=\delta_{x_u,\chi(y_p,x_p)}\delta_{y_u,\chi(y_p,x_p)}(\id_{h(\chi(y_p,x_p))}^\ast\otimes
\be_{h(\chi(y_p,x_p))}\otimes
\id_{h(\chi(y_p,x_p))})\Gamma_{h(y_p),h(x_p)}
\end{equation*}

 In Figure \ref{fig4.5} (at the end of the paper),
taking $y=x_p$ and $x=y_p$, the picture on the upper left side
represents $M$, while that on the lower right side represents
$(\mu\Delta)(p,u)$. Hence both are equal and then $\mu\Delta\doteq
\hat{\mu}(\Delta\otimes \Delta)$.
\begin{figure}[ht]
\begin{center}
\includegraphics[width=17.5cm] {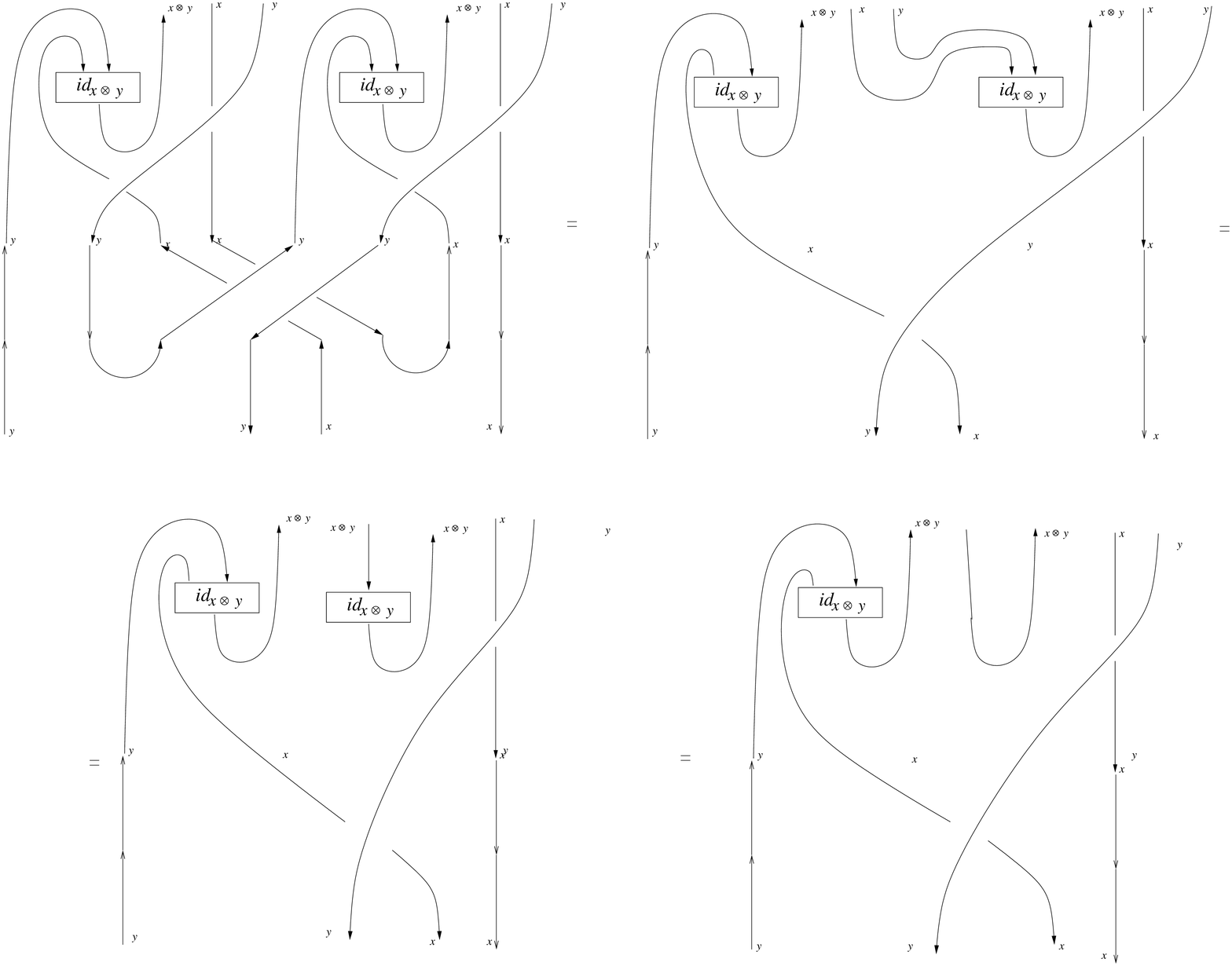}
\caption{$(\hat{\mu}(\Delta\otimes \Delta))(p,u)\doteq
(\mu\Delta)(p,u)$} \label{fig4.5}
\end{center}
\end{figure}
Finally, we have to prove that $\varepsilon$ is an algebra morphism,
that is, we have to prove that the diagram
$$\xymatrix{ \h\otimes \h
\ar[d]_{\mu}
\ar[dr]^{\varepsilon\otimes \varepsilon} & \\
\h \ar[r]^{\varepsilon} & \I  }$$ commutes. We have
\begin{equation*}
\begin{split}
(\varepsilon \mu)(u,\ast)&=\sum_w\varepsilon(\gamma(x_w,x_w),\ast)\circ \mu(u,w) \\
&=\sum_w\de_{h({x_w})}\circ
(\delta_{x_w,\chi(y_u,x_u)}\Gamma_{h(y_u),h(x_u)})\\
 &=\de_{h(\chi(y_u,x_u))}\Gamma_{h(y_u),h(x_u)}\\
 &=\de_{h(y_u)\otimes h(x_u)}\Gamma_{h(y_u),h(x_u)}
\end{split}
\end{equation*}
On the other hand, $(\varepsilon\otimes
\varepsilon)(u,\ast)=\de_{h(x_u)\otimes h(y_u)}$.

 Figure \ref{fig4.6} (at the end of the paper), taking $x=y_u$ and
 $y=x_u$ as before, shows that these two morphisms are equal.
\begin{figure}[ht]
\begin{center}
\includegraphics[width=15.5cm] {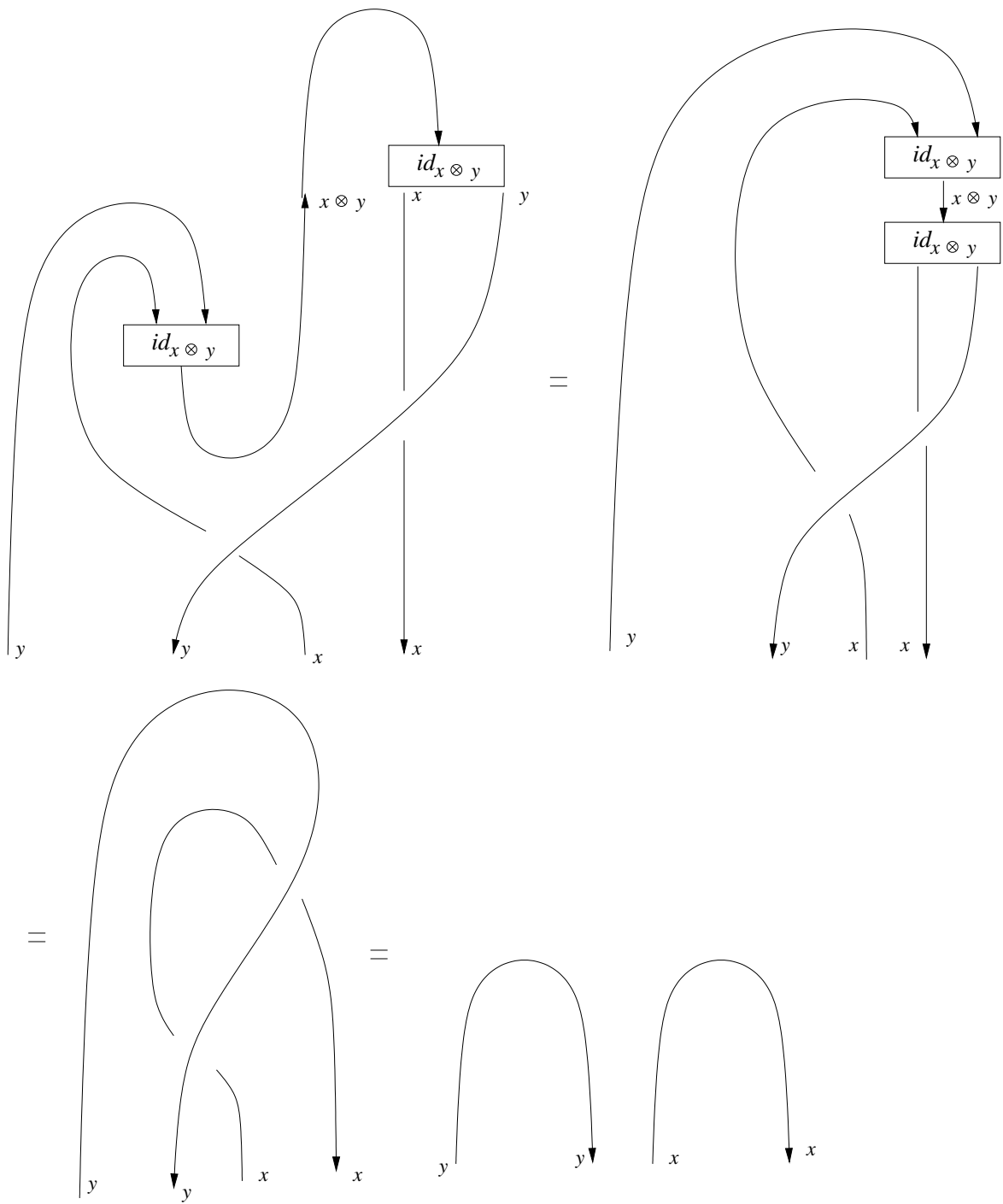}
\caption{$\de_{h_i\otimes h_j}\Gamma_{h_i,h_j}=\de_{h_j}\otimes
\de_{h_i}$}  \label{fig4.6}
\end{center}
\end{figure}
 Therefore $(\h,\mu,\eta,\Delta ,\varepsilon)$ is a bialgebra.\\
 We shall now define the action of $\h$ on the objects of $\rc$. Take the point $x_0$ of $S_h$ such that
$h(x_0)=\I$ and define $j_V(x_0)=V$, for each $V$ object of $\rc$.
Define $T:\h\otimes j_V\en j_V$, by
$$T(\gamma(\gamma(x,x),x_0),x_0)=\de_{h(x)}\otimes \id_V:h^\ast(x)\otimes h(x)\otimes V\en
V$$ where $x\in S_h$. It is not difficult to see that $T$ is indeed
an action as we defined it before. The proof of that is similar
(although easier and shorter) to the previous proofs and we omitted
it.

\end{document}